\documentclass[12pt]{article}

\usepackage{amsmath, amsthm, amssymb, amsfonts, graphicx, mathtools, color, setspace, fancyhdr}
\usepackage[margin=1in]{geometry}
\usepackage[utf8]{inputenc}
\usepackage[english]{babel}
\usepackage[svgnames]{xcolor}
\usepackage{tikz}
\usetikzlibrary{positioning, patterns}
\usepackage{setspace}
\usepackage[numbers]{natbib}
\usepackage[linkcolor=Blue,colorlinks=true,citecolor=ForestGreen,filecolor=red,urlcolor=Blue]{hyperref}
\usepackage{cleveref}[2011/12/24]
\usepackage{authblk}

\newtheorem{theorem}{Theorem}
\newtheorem{corollary}[theorem]{Corollary}
\newtheorem{lemma}[theorem]{Lemma}

\theoremstyle{definition}

\newcommand{\F}{\mathcal{F}}

\renewcommand{\P}{\mathcal{P}}
\newcommand{\Q}{\mathcal{Q}}

\newcommand{\lpt}{\mathrm{lpt}}

\newcommand{\set}[1]{\left \{#1 \right\} }

\newcommand{\ceil}[1]{\left\lceil #1 \right\rceil}

\renewcommand{\emptyset}{\varnothing}

\newcommand{\comment}[1]{}
\newcommand{\lct}{\mathrm{lct}}

\begin{document}

\author[1]{James A. Long Jr.\thanks{jalong@mix.wvu.edu}}
\author[1]{Kevin G. Milans\thanks{milans@math.wvu.edu}}
\author[2]{Andrea Munaro\thanks{a.munaro@qub.ac.uk}}
\affil[1]{Department of Mathematics, West Virginia University, USA}
\affil[2]{School of Mathematics and Physics, Queen’s University Belfast, UK}

\title{Sublinear Longest Path Transversals}
\date{\today}

\maketitle

\begin{abstract} We show that connected graphs admit sublinear longest path transversals. This improves an earlier result of Rautenbach and Sereni and is related to the fifty-year-old question of whether connected graphs admit longest path transversals of constant size. The same technique allows us to show that $2$-connected graphs admit sublinear longest cycle transversals. 
\end{abstract}

\section{Introduction}\label{intro}

A classical exercise in graph theory is to show that if $P$ and $Q$ are longest paths in a connected graph, then the vertex sets of $P$ and $Q$ have non-empty intersection (see \cite{West}, exercise 1.2.40).  In 1966, \citet{Gal68} asked whether this result could be strengthened to assert that the family of all longest paths in a connected graph $G$ has non-empty intersection.  It turns out the answer is no, as shown by \citet{Wal69} with a $25$-vertex counterexample.  A $12$-vertex counterexample, due to \citet{WV74} and independently \citet{Zam76}, is obtained from the Petersen graph by replacing one vertex $v$ with an independent set $\{v_1,v_2,v_3\}$ such that each $v_i$ becomes an endpoint of an edge incident to $v$ (see \Cref{fig:counter}).

\begin{figure}[h!]
\centering 
\includegraphics[scale=0.8]{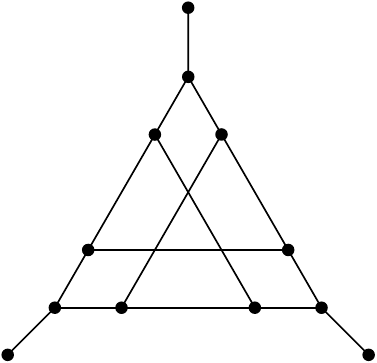}
\caption{The graph $G_0$: a $12$-vertex graph with $\lpt(G_0)=2$.}
\label{fig:counter}
\end{figure}

Since Gallai's question has a negative answer, a single vertex is generally insufficient to meet every longest path in a connected graph $G$.
A \emph{longest path transversal} in $G$ is a set of vertices that intersects every longest path.  Such a set is a transversal in the hypergraph on $V(G)$ whose edges are the vertex sets of longest paths in $G$.  Let $\lpt(G)$ be the minimum size of a longest path transversal in $G$.  The graph $G_0$ in \Cref{fig:counter} is a connected $12$-vertex graph with $\lpt(G_0)=2$. \citet{Gru74} constructed a connected $324$-vertex graph $G$ with $\lpt(G) = 3$.  Soon afterward, \citet{Zam76} found such a graph with $270$ vertices.  \Citet{Wal69} and \citet{Zam72} asked if $\lpt(G)$ is bounded for connected graphs $G$, and this remains open.  In fact, it is not known whether there is a connected graph $G$ with $\lpt(G)\ge 4$.
Let $G$ be a connected graph.  Since a connected graph does not contain vertex-disjoint longest paths, every partition of $V(G)$ into two sets has a part that contains no longest path in $G$, forcing the other part to be a longest path transversal.  Applying this to a partition of $V(G)$ into two parts of nearly equal size gives $\lpt(G)\le \ceil{n/2}$ when $G$ is an $n$-vertex connected graph.  It is not too difficult to improve this argument to obtain $\lpt(G)\le\ceil{n/4}$.  \citet{RS14} showed that $\lpt(G) \leq \lceil \frac{n}{4} - \frac{n^{2/3}}{90}\rceil$ for every connected $n$-vertex graph $G$.  We show that $\lpt(G)\le 8n^{3/4}$ when $G$ is an $n$-vertex connected graph, implying that connected graphs have sublinear longest path transversals.

Let $\lct(G)$ be the minimum size of a set of vertices $S$ such that $S$ intersects every longest cycle in $G$. Analogously to the case of longest paths in $1$-connected graphs, every pair of longest cycles in a $2$-connected graph intersect.  The Petersen graph $G$ is $2$-connected and $\lct(G)=2$.  With no connectivity assumptions, \citet{Tho78} showed that $\lct(G) \leq \lceil n/3\rceil$ for each $n$-vertex graph $G$.  The bound is sharp when $G$ is a disjoint union of triangles and nearly sharp in the $1$-connected case when $G$ is obtained from a star with $(n-1)/3$ leaves by replacing each leaf with a triangle.  On the other hand, \citet{RS14} proved that if $G$ is $2$-connected, then $\lct(G) \leq \lceil \frac{n}{3} - \frac{n^{2/3}}{36} \rceil$.  We show that $\lct(G) \leq 20n^{3/4}$ when $G$ is $2$-connected (\Cref{cor:lpt-lct}).  

The problems of finding small longest path transversals and small longest cycle transversals are special cases of a general problem that we aim to address.  Given a multigraph $F$ and an edge $e\in E(F)$ with endpoints $u$ and $v$, the \emph{subdivision operation} produces a new multigraph $F'$ in which $e$ is replaced by a path $uwv$ through a new vertex $w$ in $F'$.  A \emph{subdivision} of $F$ is a graph obtained from $F$ via a sequence of zero or more subdivision operations.  For a multigraph $R$ and a graph $G$, an $R$-subdivision in $G$ is a subgraph of $G$ isomorphic to a subdivision of $R$.  We ask for a small set of vertices in $G$ that intersects every
$R$-subdivision in $G$ of maximum size.  The cases of longest path transversals and longest cycle transversals arise as $R=P_2$ and $R=C_2$ (the multigraph $2$-vertex cycle), respectively.  We prove that for each connected multigraph $R$, if the family $\F$ of maximum $R$-subdivisions in $G$ is pairwise intersecting, then $\F$ admits a transversal of size at most $Cn^{3/4}$, where $C$ is a constant depending on $R$. 

\newcommand{\vep}{\varepsilon}
\newcommand{\mst}{\tau}

\section{Maximum subdivision transversals}\label{sec:mst}
Let $R$ be a multigraph.  Recall that an $R$-subdivision in $G$ is a subgraph of $G$ isomorphic to a subdivision of $R$, and a \emph{maximum $R$-subdivision} is an $R$-subdivision $F$ in $G$ that maximizes $|V(F)|$.  An \emph{$R$-transversal} of $G$ is a set of vertices intersecting each maximum $R$-subdivision.  Let $\mst_R(G)$ be the minimum size of an $R$-transversal in $G$.

Given sets of vertices $X$ and $Y$ of $G$, an \emph{$(X,Y)$-separator} is a set of vertices $S$ such that no path in $G-S$ has one endpoint in $X$ and the other endpoint in $Y$.  We allow an $(X,Y)$-separator to contain vertices in $X$ and $Y$.  An \emph{$(X,Y)$-connector} is a collection of vertex-disjoint paths $\set{P_1, \ldots, P_k}$ such that each $P_i$ has one endpoint in $X$, the other endpoint in $Y$, and the interior vertices of $P_i$ are outside $X \cup Y$.  A variant of Menger's Theorem asserts that the minimum size of an $(X,Y)$-separator equals the maximum size of an  $(X,Y)$-connector (see, e.g., Theorem 3.3.1 in \citep{Die}).

Our next result shows that when the maximum $R$-subdivisions in a graph $G$ pairwise intersect, $G$ has sublinear $R$-transversals.  We make no attempt to optimize the multiplicative constant $8$ or the dependence on $m$.

\begin{theorem}\label{thm:small-trans}
Let $R$ be a connected $m$-edge multigraph with $m\ge 1$ and let $G$ be an $n$-vertex graph.  If the maximum $R$-subdivisions in $G$ pairwise intersect, then $\mst_R(G) \le 8m^{5/4}n^{3/4}$.
\end{theorem}

\begin{proof}
Let $m=|E(R)|$ and let $\vep = 2(m/n)^{1/4}$.  We may assume that $m\le n$, since otherwise we may take $V(G)$ as our $R$-transversal.  Let $\F$ be the family of maximum $R$-subdivisions in $G$.  An \emph{$\vep$-partial transversal} is a triple $(H,X,Y)$ such that $H$ is a subgraph of $G$, $X=V(G)-V(H)$, $Y \subseteq X$ with $|Y|\le \vep|X|$, and each $F\in\F$ is a subgraph of $H$ or contains a vertex in $Y$.  Given an $\vep$-partial transversal $(H,X,Y)$, we either obtain an $\vep$-partial transversal $(H',X',Y')$ with $|V(H')|<|V(H)|$ or we produce an $R$-transversal with at most $8m^{5/4}n^{3/4}$ vertices.  Starting with $(H,X,Y)=(G,\emptyset,\emptyset)$ and iterating gives the result.

Let $(H,X,Y)$ be an $\vep$-partial transversal, and let $\F_0$ be the set of $F\in\F$ such that $F$ is a subgraph of $H$.  We may assume that $H$ contains vertex-disjoint paths $P_1$ and $P_2$ each of size $\ceil{\vep n}$. Otherwise, every path in $H$ has size less than $2\ceil{\vep n}$, and so each $F\in\F_0$ has at most $2m\ceil{\vep n}$ vertices. 
Since $\F_0$ is pairwise intersecting, we have that $V(F)\cup Y$ is an $R$-transversal for each $F\in \F_0$.  It follows that $\mst_R(G)\le |Y| + 2m\ceil{\vep n} \le \vep n + 2m\ceil{\vep n} \le (2m+1)\vep n + 2m \le (2m+2)\vep n \le 4m\vep n = 8m^{5/4}n^{3/4}$. 

Suppose that $H$ has a $(V(P_1), V(P_2))$-separator $S$ of size at most $\vep^2 n$.  Since graphs in $\F_0$ are connected, each $F\in \F_0$ has a vertex in $S$ or is contained in some component of $H-S$.  Also, since $\F_0$ is pairwise intersecting, at most one component $H'$ of $H-S$ contains graphs in $\F_0$.  Since $S$ is a separator, $H'$ is disjoint from at least one of $\set{P_1,P_2}$.  With $X'=V(G)-V(H')$ and $Y'=Y\cup S$, we have $|X'| - |X| \ge \vep n$ and $|Y'| = |Y| + |S| \le \vep|X| + \vep^2 n \le \vep|X| + \vep (|X'|-|X|) \le \vep|X'|$.  It follows that $(H',X',Y')$ is an $\vep$-partial transversal.  Also $|V(H')|<|V(H)|$ since $|X'| > |X|$.

Otherwise, by Menger's Theorem, $H$ has a $(V(P_1),V(P_2))$-connector $\P$ with $|\P| \ge \vep^2 n$.  Let $\P'$ be the set of paths in $\P$ of size at most $2/\vep^2$.  Note that $|\P'| \ge |\P|/2$, or else $\P$ has at least $(\vep^2 n)/2$ paths of size more than $2/\vep^2$, contradicting that the paths in $\P$ are disjoint.  So we have $|\P'| \ge |\P|/2 \ge (\vep^2/2)n = 2m^{1/2}n^{1/2} \ge 2$.  Combining $P_1$ with two paths in $\P'$ whose endpoints in $V(P_1)$ are as far apart as possible and a segment of $P_2$ gives a cycle $C_0$ such that $(\vep^2/2)n \le |V(C_0)| \le  2\ceil{\vep n} + 4/\vep^2 - 4 \le 2\vep n + 4/\vep^2$, where the lower bound counts vertices in $V(P_1)\cap V(C_0)$ and the upper bound counts at most $2\ceil{\vep n}$ vertices in $(V(P_1)\cup V(P_2)) \cap V(C_0)$, at most $4/\vep^2$ vertices on the paths in $\P'$ linking $P_1$ and $P_2$, and observing that the $4$ endpoints of the linking paths are counted twice.  

Let $C$ be a longest cycle in $H$ subject to $|V(C)| \le 2\vep n + 4/\vep^2$, let $\ell = |V(C)|$,  and note that $\ell \ge |V(C_0)| \ge (\vep^2/2)n$.  If $V(C)$ intersects each subgraph in $\F_0$, then $Y\cup V(C)$ witnesses $\mst_R(G)\le |V(C)| + |Y| \le (2\vep n + 4/\vep^2) + \vep n = 3\vep n + (n/m)^{1/2} < 8m^{5/4}n^{3/4}$.  Otherwise, choose $F\in\F_0$ that is disjoint from $C$.  We may assume $|V(F)| \ge \ell$, or else $Y\cup V(F)$ witnesses that $\mst_R(G) \leq |V(F)| + |Y| < (2\vep n + 4/\vep^2) + \vep n < 8m^{5/4}n^{3/4}$.

If $H$ has a $(V(C),V(F))$-separator $T$ of size at most $\vep \ell$, then we obtain an $\vep$-partial transversal as follows.  At most one component $H'$ of $H-T$ contains graphs in $\F_0$.  Let $X'=V(G)-V(H')$ and let $Y'=Y\cup T$.  Since $H'$ is disjoint from one of $\set{C,F}$, it follows that $|X'| - |X| \ge \ell$.  We compute $|Y'| = |Y| + |T| \le \vep|X| + \vep \ell \le \vep|X| + \vep (|X'| - |X|) \le \vep |X'|$.  Hence $(H',X',Y')$ is an $\vep$-partial transversal with $|V(H')|<|V(H)|$.  

Otherwise, $H$ has a $(V(C), V(F))$-connector $\Q$ with $|\Q| \ge \vep \ell$.  We use $\Q$ to obtain a contradiction.  For $e\in E(R)$, let $Q_e$ be the path in $F$ corresponding to $e$, and let $\Q_e$ be the set of paths in $\Q$ which have an endpoint in $Q_e$.  Since $|E(R)| = m$, it follows that $|\Q_e| \ge |\Q|/m \ge \vep \ell/m$ for some edge $e\in E(R)$.  Let $\Q'$ be the set of paths in $\Q_e$ of size at most $\frac{2mn}{\vep \ell}$, and note that $|\Q'|\ge |\Q_e|/2 \ge \frac{\vep \ell}{2m}$, or else $\Q_e$ has at least $\frac{\vep \ell}{2m}$ paths of size more than $\frac{2mn}{\vep \ell}$, a contradiction.  The endpoints of paths in $\Q'$ divide $Q_e$ into $|\Q'|-1$ edge-disjoint subpaths.  Choose $Q_1,Q_2\in\Q'$ to minimize the length of such a subpath $Q_0$ of $Q_e$, and note that $Q_0$ has length at most $\frac{n-1}{|\Q'|-1}$; see Figure~\ref{fig:detours}. Since $m\le n$, we have $2m \le 2m^{3/4}n^{1/4} = \frac{\vep^3}{4} n \le \frac{\vep\ell}{2}$, and hence $\frac{n-1}{|\Q'| - 1} < \frac{n}{\frac{\vep \ell}{2m} - 1} = \frac{2mn}{\vep \ell - 2m} \le \frac{4mn}{\vep \ell}$.

\begin{figure}
\begin{center}
    \begin{tikzpicture}[
            vertex/.style={circle, inner sep=0pt, minimum size=1ex, fill=black!80}            
        ]
        
        \draw (0,2) circle (1cm) ++(1.1,0) node[anchor=west] {$C$} ;
        
        \path (0,2) ++(225:1cm) node[vertex] (Q1C) {} ;
        \path (0,2) ++(315:1cm) node[vertex] (Q2C) {} ;
        
        \begin{scope}[shift={(0,-2)}]
        \fill[black!15] (0, 0) ellipse (3.5cm and 1.25cm) ;
        
        \node[vertex] at (-2.5, 0) (QeA) {} ;
        
        \node[vertex] at (2.2, 0.4) (QeB) {} ;
        
        \draw (QeA) -- node[pos=0.5, anchor=north west] {$Q_e$} (QeB) ;
        
        \path (QeA) -- node[pos=0.15,vertex] (Q1F) {} (QeB) ;
        
        \path (QeA) --
        node[pos=0.4,vertex] (Q2F) {} (QeB) ;
        
        \node at (3, 0) {$F$} ;
        \end{scope}
        
        \draw (Q1F) edge[left] node {$Q_1$} (Q1C) ;
        \draw (Q2F) edge[right] node {$Q_2$} (Q2C) ;
        
        \path (Q1F) -- node[anchor=south] {$Q_0$} (Q2F) ;
        
        \draw[dashed,line width=1.5] (0,2) ++(225:1cm) arc (225:-45:1cm) node[pos=0.3,anchor=south east,inner sep=0pt] {$W$} ;
        
        \draw[line width=1.5] (Q2F) -- (Q1F) -- (Q1C) arc (225:315:1cm) -- (Q2F) ;
    \end{tikzpicture}
\end{center}
\caption{$(V(C),V(F))$-connector case.  The subpath $W$ of the cycle $C$ is dashed, and the cycle $D$ is displayed in bold.}\label{fig:detours}
\end{figure}

The endpoints of $Q_1$ and $Q_2$ on $C$ partition $C$ into two subpaths; let $W$ be the longer subpath. If $|E(W)| \ge |E(Q_0)|$, then we would obtain a larger $R$-subdivision by using $Q_1$, $W$, and $Q_2$ to bypass $Q_0$. Since $F$ is a maximum $R$-subdivision, we have $|E(W)|<|E(Q_0)|$.  Therefore using $Q_1$, $Q_0$, and $Q_2$ to bypass $W$ gives a cycle $D$ with $|E(D)|>|E(C)|$.  By the extremal choice of $C$, it follows that $|V(D)|> 2\vep n + 4/\vep^2$. On the other hand, $|V(D)| = |E(D)| \le \frac{\ell}{2} + |E(Q_1)| + |E(Q_0)| + |E(Q_2)| \le \frac{\ell}{2} + \frac{2mn}{\vep \ell} + \frac{4mn}{\vep \ell} + \frac{2mn}{\vep \ell} = \frac{\ell}{2} + \frac{8mn}{\vep \ell}$.

Therefore $2\vep n+\frac{4}{\vep^2} < |V(D)| \le \frac{\ell}{2} + \frac{8mn}{\vep \ell} \le \vep n + \frac{2}{\vep^2} + \frac{8mn}{\vep \ell} \le  \vep n + \frac{2}{\vep^2} + \frac{16m}{\vep^3}$, where the last inequality uses $\ell \ge (\vep^2/2) n$.  Simplifying gives $\vep n < \frac{16m}{\vep^3} - \frac{2}{\vep^2} < \frac{16m}{\vep^3}$, and this inequality is violated when $\vep \ge (16m/n)^{1/4}$.  
\end{proof}

Applying Theorem~\ref{thm:small-trans}, we obtain the following corollary.

\begin{corollary}\label{cor:lpt-lct}
Let $G$ be an $n$-vertex graph.  If $G$ is connected, then $\lpt(G)\le 8n^{3/4}$.  If $G$ is $2$-connected, then $\lct(G)\le 20n^{3/4}$.
\end{corollary}
\begin{proof}
When $R=P_2$, an $R$-transversal is a longest path transversal.  It is well known that if $G$ is connected, then the longest paths pairwise intersect.  By Theorem~\ref{thm:small-trans}, we have $\lpt(G)=\mst_R(G)\le 8n^{3/4}$.  

Similarly, when $R=C_2$, an $R$-transversal is a longest cycle transversal.  If $G$ is $2$-connected, then the longest cycles pairwise intersect.  By Theorem~\ref{thm:small-trans}, we have $\lct(G) = \mst_R(G)\le 8\cdot 2^{5/4}\cdot n^{3/4}\le 20n^{3/4}$.
\end{proof}

We do not know whether the assumption in  Theorem~\ref{thm:small-trans} that $R$ is connected is necessary to obtain sublinear $R$-transversals.  To obtain analogues of \Cref{cor:lpt-lct} for general $R$, we show that the maximum $R$-subdivisions pairwise intersect when the connectivity of $G$ is sufficiently large. Recall that a graph $G$ is \emph{$k$-connected} if $|V(G)| > k$ and $G - S$ is connected for each $S \subseteq V(G)$ with $|S| < k$. Moreover, the \emph{connectivity} of $G$, denoted $\kappa(G)$, is the maximum $k$ such that G is $k$-connected.

\begin{lemma}\label{lem:high-connect-intersect}
Let $R$ be a connected $m$-edge multigraph with $m\ge 1$.  If $\kappa(G) > m^2$, then the maximum $R$-subdivisions in $G$ are pairwise intersecting.
\end{lemma}
\begin{proof}
Suppose for a contradiction that $G$ has disjoint maximum $R$-subdivisions $F_1$ and $F_2$, and let $k = |V(F_1)| = |V(F_2)|$.  By Menger's Theorem, there is an $(V(F_1),V(F_2))$-connector $\P$ with $|\P| = \min\{k,m^2+1\}$. 
If $|\P| = k$, then every vertex in $F_1$ is an endpoint of a path in $\P$, and we obtain an $R$-subdivision of size more than $k$ by replacing an edge $uv\in E(F_1)$ with a path in $\P$ having $u$ as an endpoint, a path in $\P$ having $v$ as an endpoint, and an appropriate path in the connected subgraph $F_2$.

So we may assume $|\P|=m^2 + 1$.  For each $e\in E(R)$, let $F_i(e)$ be the path in $F_i$ corresponding to $e$.  Since $R$ has no isolated vertices, we may associate each $P\in\P$ with an ordered pair of edges $(e_1,e_2)\in (E(R))^2$ such that $P$ has its endpoint in $F_1$ in $F_1(e_1)$ and its endpoint in $F_2$ in $F_2(e_2)$. Since $|\P|>m^2$, some pair $(e_1,e_2)$ is associated with distinct paths $P,Q\in \P$.  Let $W_i$ be the subpath of $F_i(e_i)$ whose endpoints are in $V(P)\cup V(Q)$.  If $|E(W_1)|\ge |E(W_2)|$, then we modify $F_2$ to obtain a larger $R$-subdivision by using $P$, $W_1$, and $Q$ to bypass $W_2$.  Similarly, if $|E(W_2)| \ge |E(W_1)|$, then we modify $F_1$ to obtain a larger $R$-subdivision by using $P$, $W_2$, and $Q$ to bypass $W_1$.
\end{proof}

\begin{corollary}
Let $R$ be a connected $m$-edge multigraph.  If $G$ is an $n$-vertex graph with $\kappa(G)>m^2$, then $\mst_R(G)\le 8m^{5/4}n^{3/4}$.
\end{corollary}

As it is not known whether there exists a connected graph $G$ with $\lpt(G) > 3$, reducing the gap between our sublinear upper bound on $\lpt(G)$ and the constant lower bound remains a major open problem in the area of longest path transversals. 

\subsection*{Acknowledgement}
The authors greatly appreciate the careful comments of an anonymous referee.

\bibliographystyle{plainnat}
\bibliography{citations}

\begin{thebibliography}{10}
\providecommand{\natexlab}[1]{#1}
\providecommand{\url}[1]{\texttt{#1}}
\expandafter\ifx\csname urlstyle\endcsname\relax
  \providecommand{\doi}[1]{doi: #1}\else
  \providecommand{\doi}{doi: \begingroup \urlstyle{rm}\Url}\fi

\bibitem[Diestel(2005)]{Die}
R.~Diestel.
\newblock \emph{{G}raph {T}heory}.
\newblock Graduate Texts in Mathematics. Springer, 2005.

\bibitem[Gallai(1968)]{Gal68}
T.~Gallai.
\newblock Problem 4.
\newblock In P.~Erd\H{o}s and G.~Katona, editors, \emph{Theory of Graphs,
  Proceedings of the Colloquium Held at Tihany, Hungary, September 1966}, page
  362. Academic Press, New York, 1968.

\bibitem[Gr\"{u}nbaum(1974)]{Gru74}
B.~Gr\"{u}nbaum.
\newblock Vertices missed by longest paths or circuits.
\newblock \emph{Journal of Combinatorial Theory, Series A}, 17\penalty0
  (1):\penalty0 31--38, 1974.

\bibitem[Rautenbach and Sereni(2014)]{RS14}
D.~Rautenbach and J.-S. Sereni.
\newblock Transversals of longest paths and cycles.
\newblock \emph{SIAM Journal on Discrete Mathematics}, 28\penalty0
  (1):\penalty0 335--341, 2014.

\bibitem[Thomassen(1978)]{Tho78}
C.~Thomassen.
\newblock Hypohamiltonian graphs and digraphs.
\newblock In Y.~Alavi and D.R. Lick, editors, \emph{Theory and Applications of
  Graphs}, pages 557--571. Springer Berlin Heidelberg, 1978.

\bibitem[Walther(1969)]{Wal69}
H.~Walther.
\newblock {\"U}ber die nichtexistenz eines knotenpunktes, durch den alle
  längsten wege eines graphen gehen.
\newblock \emph{Journal of Combinatorial Theory}, 6\penalty0 (1):\penalty0
  1--6, 1969.

\bibitem[Walther and Voss(1974)]{WV74}
H.~Walther and H.-J. Voss.
\newblock \emph{\"{U}ber Kreise in Graphen}.
\newblock Deutscher Verlag der Wissenschaften, 1974.

\bibitem[West(2001)]{West}
D.~B. West.
\newblock \emph{{I}ntroduction to {G}raph {T}heory}.
\newblock Prentice Hall, 2nd edition, 2001.

\bibitem[Zamfirescu(1972)]{Zam72}
T.~Zamfirescu.
\newblock A two-connected planar graph without concurrent longest paths.
\newblock \emph{Journal of Combinatorial Theory, Series B}, 13\penalty0
  (2):\penalty0 116--121, 1972.

\bibitem[Zamfirescu(1976)]{Zam76}
T.~Zamfirescu.
\newblock On longest paths and circuits in graphs.
\newblock \emph{Mathematica Scandinavica}, 38:\penalty0 211--239, 1976.

\end{thebibliography}

\end{document}